\documentclass{article}

\usepackage{amssymb}
\usepackage{amsmath}
\usepackage{amsfonts}
\usepackage{amsthm}
\usepackage{mathabx}
\usepackage{graphicx}
\usepackage{cite}
\usepackage[breaklinks=true]{hyperref}
\usepackage{graphicx}
\usepackage{tikz}
\usepackage{placeins}
\usepackage{float}
\usepackage{pstricks}
\usepackage{tikz-cd}
\usepackage[capitalize]{cleveref}
\usepackage{bm}
\usepackage{mathtools}
\usepackage[shortcuts]{extdash}

\usetikzlibrary{arrows,chains,matrix,positioning,scopes}

\newtheorem{thm}{Theorem}[section]
\newtheorem{theorem}[thm]{Theorem}
\newtheorem{lem}[thm]{Lemma}

\newtheorem{prop}[thm]{Proposition}
\newtheorem{cor}[thm]{Corollary}

\newtheorem{question}{Question}
\newtheorem{thmm}{Theorem}

\newtheorem{thmk}{Theorem}

\newtheorem{iinnercustomthm}{Corollary}
\newenvironment{customcor}[1]
{\renewcommand\theiinnercustomthm{#1}\iinnercustomthm}
{\endiinnercustomthm}

\theoremstyle{definition}
\newtheorem{definition}[thm]{Definition}

\theoremstyle{remark}
\newtheorem{rem}[thm]{Remark}

\newtheorem*{prop proof}{Proof of Proposition}

\newcommand{\RR}{\mathbb{R}}      
\newcommand{\ZZ}{\mathbb{Z}}        
\newcommand{\EE}{\mathbb{E}}
\renewcommand{\SS}{\mathbb{S}}
\newcommand{\CN}{\mathcal{N}}
\newcommand{\lk}{{\rm{Lk}}}
\newcommand{\st}{{\rm{St}}}
\newcommand{\xcplx}{{\mathbf{X}_{\Gamma_A,\Gamma_B}}} 
\newcommand{\xcplxij}[2]{{\mathbf{X}_{\Gamma_{A_{#1#2}},\Gamma_{B_{#1#2}}}}} 

\newcommand{\xbarcplx}{{\overline{\mathbf{X}}}} 
\newcommand{\xtilde}{{\widetilde{\mathbf{X}}}}

\DeclarePairedDelimiter\abs{\lvert}{\rvert}		
\DeclarePairedDelimiter\paren{(}{)}			
\DeclarePairedDelimiter\braces{\{}{\}}			

\newcommand{\dotsqcup}{\mathbin{\mathaccent\cdot\sqcup}}

	\newcommand{\bigparen}[1]{\paren[\big]{#1}}

\newcommand{\interior}[1]{   
	\mathring{#1}
} 		

\newcommand{\CH}{\mathcal{H}}
\newcommand{\CU}{\mathcal{U}}
\newcommand{\fks}{\hat{\mathfrak{s}}}
\newcommand{\fkt}{\hat{\mathfrak{t}}}
\newcommand{\hpl}{\hat{\mathfrak{h}}}
\newcommand{\Hpl}{\widehat{\CH}}
\newcommand{\hsp}{\mathfrak{h}}

\newcommand{\fkb}{\hat{\mathfrak{b}}}
\newcommand{\fka}{\hat{\mathfrak{a}}}
\newcommand{\fkh}{\hat{\mathfrak{h}}}

\newcommand{\bigbrace}[1]{\braces[\big]{#1}}	
	
\newcommand{\bigmid}{\mathrel{\big|}}			

\newcommand{\Hsp}{\CH}

\newcommand{\hexagon}{
	\begin{tikzpicture}
	\draw (0ex,0ex) -- (0.5ex,0.87ex) -- (1.5ex,0.87ex) -- (2ex,0ex) -- (1.5ex,-0.87ex) -- (0.5ex,-0.87ex) -- (0ex,0ex);
	\end{tikzpicture}
}

\newcommand{\boldhexagon}{
	\begin{tikzpicture}
	\draw (0,0) -- (0.5ex,0.87ex) -- (1.5ex,0.87ex) -- (2ex,0ex) -- (1.5ex,-0.87ex) -- (0.5ex,-0.87ex) -- (0ex,0ex);
	\filldraw (0,0)[black] circle (1pt); 
	\filldraw (1.5ex,0.87ex) [black] circle (1pt);
	\filldraw (1.5ex,-0.87ex) [black] circle (1pt);
	\end{tikzpicture}
}

\numberwithin{equation}{section}

\title{Hyperbolic groups with almost finitely presented subgroups}

\author{Robert Kropholler \\ {\small (with an appendix by Robert Kropholler and  Federico Vigolo)}}

\begin{document}
\maketitle

\begin{abstract}
	In this paper we create many examples of hyperbolic groups with subgroups satisfying interesting finiteness properties. We give the first examples of subgroups of hyperbolic groups which are of type $FP_2$ but not finitely presented. We give uncountably many groups of type $FP_2$ with similar properties to those subgroups of hyperbolic groups. Along the way we create more subgroups of hyperbolic groups which are finitely presented but not of type $FP_3$.
\end{abstract}
\section{Introduction}

Given a class of groups $\mathcal{C}$ it is natural to ask what finitely generated subgroups of $\mathcal{C}$-groups are like. One may hope that these are already $\mathcal{C}$-groups as is the case for the classes of free groups, surface groups and 3-manifold groups. One can also ask if groups in the class are coherent, i.e. every finitely generated subgroup is finitely presented. For the class of Artin groups, coherency fails \cite{bestvina_morse_1997} and the subgroups form an interesting class \cite{wise_structure_2011}.

A very well studied class of interest in geometric group theory is the class of hyperbolic groups. For special subclasses of hyperbolic groups we obtain no new information. However, in \cite{rips_subgroups_1982} it is shown that there are finitely generated subgroups of hyperbolic groups which are not finitely presented and therefore not hyperbolic. One could then ask if finitely presented subgroups behave in a nicer way although this is shown to not be the case \cite{brady_branched_1999,kropholler_finitelypresented_2018,lodha_hyperbolic_2017other}.

Finite generation was not enough to guarantee that subgroups were finitely presented. One could look for a stronger property between finite generation and finite presentability. One such property is being of type $FP_2$. For simplicity, in this paper we shall only consider type $FP_2$ over $\ZZ$. 

\begin{definition}
	A group $G$ is of {\em type $FP_2$} if there is a partial resolution $$P_2\to P_1\to P_0\to \ZZ,$$  where each $P_i$ is a finitely generated projective $\ZZ G$ module. 
\end{definition}

One can think of this as a homological version of finite presentability. Indeed, it is equivalent to finite generation of the relation module. Until \cite{bestvina_morse_1997} it was unknown whether being of type $FP_2$ was equivalent to finite presentation. It was shown that in subgroups of RAAGs one could obtain groups which were of type $FP_2$ but not finitely presented. 

One can attempt to find such groups inside hyperbolic groups. A result showing this may not be possible is the following 

\begin{thm}\cite{gersten_subgroups_1996}
	Let $G$ be a hyperbolic group of cohomological dimension 2. If $H<G$ is of type $FP_2$, then $H$ is hyperbolic. In particular, $H$ is finitely presented. 
\end{thm}

In this paper we show that this phenomenon is special to cohomological dimension 2. We find hyperbolic groups of cohomological dimension 3 containing subgroups which are of type $FP_2$ but not finitely presented. Namely, we prove the following. 

\begin{thmk}\label{thm:subhypfp2}
	There exists a hyperbolic group $G$ and a homomorphism $\phi\colon G\to \ZZ$ such that $\ker(\phi)$ is of type $FP_2$ but not finitely presented. 
\end{thmk}

It should be noted that the subgroups constructed here are not of type $FP_3$. We leave the following as an open question. 

\begin{question}
	Let $n>2$ be an integer or $\infty$. Is there a subgroup of a hyperbolic group which is of type $FP_{n}$ but is not finitely presented?
\end{question}

One could attempt this problem in reverse, i.e. one could look for groups which have interesting properties and attempt to embed them in hyperbolic groups. To do this one would need to start with a full list of obstructions to embedding in a hyperbolic group. For instance, subgroups of hyperbolic groups cannot contain Baumslag-Solitar groups or infinite torsion groups. 

\begin{definition}
	Let $m, n$ be non-zero integers. The {\em Baumslag-Solitar group} is defined by the presentation $BS(m, n) = \langle x, y\mid y^{-1}x^my = x^n\rangle.$
\end{definition}

We see that $BS(1, 1)$ is isomorphic to $\ZZ^2$. 

Using techniques from \cite{leary_uncountably_2015other} we construct uncountably many groups which do not contain $BS(m, n)$ for any $m, n$.  

\begin{thmk}
	There are uncountably many groups which are of type $FP_2$ and do not contain any subgroups isomorphic to $BS(m, n)$ or infinite torsion groups. 
\end{thmk}

These groups are created by using a specific group $H$ constructed for the proof of Theorem \ref{thm:subhypfp2}. This group is finitely generated but infinitely presented and has a presentation of the form $\langle S\mid U\rangle$, where $S$ is finite and $U$ is infinite. To each $Z\subseteq U$ we associate the group $H(Z) = \langle S\mid Z\rangle$. We show that for each $Z'\subseteq U$ there are only countably many groups $H(Z)$ which are isomorphic to $H(Z')$. We show that among these isomorphism classes uncountably many are of type $FP_2$. 

Only countably many of the above groups can be embedded into finitely presented groups, hence only countably many can be embedded into hyperbolic groups. However, these groups contain none of the above mentioned obstructions. 

\begin{question}
	Which of the groups $H(Z)$ embed into hyperbolic groups?
\end{question} 

Acknowledgments: I thank Noel Brady, Max Forester, Ignat Soroko and an anonymous referee for helpful comments. 

\section{Preliminaries}

\subsection{CAT(0) Cube Complexes}\label{sec:flymaps}

Here we discuss the cube complexes we are interested in. For full details on cube complexes see \cite{sageev_cat0_2014}.

We use the construction from \cite{kropholler_new_2018}. In \cite{kropholler_new_2018}, cube complexes are constructed using two flag complexes $\Gamma_A$ and $\Gamma_B$ with $n$-partite structures.

\begin{definition}
	An {\em $n$-partite struture} on a flag complex $L$ is a partition of the vertices of $L$ into sets $V_1, \dots, V_n$ such that the natural map $V(L)\to V_1\ast \dots\ast V_n$ extends to an embedding of $L\to V_1\ast \dots\ast V_n$. 
\end{definition}

A flag complex with a $n$-partite structure has dimension at most $n-1$. 
For $n=2$ we recover the definition of a bipartite graph. In this paper, we will focus on the case $n=3$ when the structure is {\em tripartite.} In this case, we give an account of the required material from \cite{kropholler_new_2018}.

Let $\Gamma_A\subseteq A_1\ast A_2\ast A_3$ and $\Gamma_B\subseteq B_1\ast B_2\ast B_3$ be flag complexes with tripartite structures. Set $K$ equal to the non-positively curved cube complex $\prod_{i=1}^3 A_i*B_i$. Given a vertex $v = (v_1, v_2, v_3)\in K$ we can assign two sets, $\Delta_A = \{v_i\in A_i\}$ and $\Delta_B = \{v_i\in B_i\}$. Let $$V = \{v\in K \mid \Delta_A \mbox{ is a simplex of }\Gamma_A, \Delta_B \mbox{ is a simplex of }\Gamma_B\}.$$

We define $\xcplx$ to be the maximal subcomplex of $K$ with vertex set $V$. 

The following can be found in \cite{kropholler_new_2018}. 

\begin{lem}\label{lem:links.in.xcplx}
	The link of the vertex $(v_1, v_2, v_3)$ in $\xcplx$ is the join of the links of $\Delta_A$ and $\Delta_B$ i.e.
	\[
	\lk\bigg((v_1, v_2, v_3),\xcplx\bigg)=\lk(\Delta_A,\Gamma_A)*\lk(\Delta_B,\Gamma_B).
	\]
\end{lem} 

Since links of simplices in a flag complex are flag complexes and the join of two flag complexes is a flag complex. Lemma \ref{lem:links.in.xcplx} shows that the link of a vertex in $\xcplx$ is a flag complex and we obtain the following. 

\begin{cor}\label{cor:flag.implies.cat0.CLCC}
	$\xcplx$ is a non-positively curved cube complex.
\end{cor}

These complexes form the base of our construction. We take branched covers to eliminate any isometrically embedded flat planes. To show that there are no flat planes in the universal cover, we use the notion of fly maps. The details are included in the appendix. 

We require the following key result from the appendix. 

\begin{customcor}{\ref{cor:flylinks}}
	A fly map $f^x$ induces for every cube $c\in \CN\big(i(\EE^k)\big)$ an embedding $\lk\big(c, \CN\big(i(\EE^2)\big)\big)\to S^0\ast S^0\ast S^0$.  
\end{customcor}

Fly maps allow us to better understand flat planes in CAT(0) cube complexes, which by Bridson's flat plane theorem are the only obstruction to hyperbolicity: 

\begin{thm}\cite{bridson_existence_1995}\label{thm:flatplane}
	A compact CAT(0) space $X$ has hyperbolic fundamental group if and only if there are no isometric embeddings of $\EE^2$ to $\tilde{X}$. 
\end{thm}

\subsection{Branched Covers}

To obtain hyperbolic cube complexes, we begin with a cube complex containing flat planes and take branched covers. In the setting of CAT(0) cube complexes this was achieved by Brady \cite{brady_branched_1999} and used to create subgroups of hyperbolic groups with interesting finiteness properties. It has also been used in \cite{kropholler_finitelypresented_2018,kropholler_almost_2018} to create other groups with interesting finiteness properties. 

\begin{definition}\label{def:branchinglocus}
	A {\em branching locus} $L$ in a non-positively curved cube complex $K$ is a subcomplex satisfying the following two conditions. 
	
	\begin{itemize}
		\item $L$ is a locally isometrically embedded subcomplex of $K$.
		\item $\lk(c, K)\smallsetminus\lk(c, L)$ is non-empty and connected for all cubes $c\in L$. 
	\end{itemize}
\end{definition}

The first condition is required to prove that non-positive curvature is preserved when taking branched covers. The second is a reformulation of the classical requirement that the branching locus has codimension 2 in the theory of branched covers of manifolds; it ensures that the trivial branched covering of $K$ is $K$. 

\begin{definition}
	A {\em branched cover} $\widehat{K}$ of a non-positively curved cube complex $K$ over the branching locus $L$ is the result of the following process.
	\begin{enumerate}
		\item Take a finite covering $\overline{K\smallsetminus L}$ of $K\smallsetminus L$.
		\item Lift the piecewise Euclidean metric locally and consider the induced path metric on $\overline{K\smallsetminus L}$.
		\item Take the metric completion $\widehat{K}$ of $\overline{K\smallsetminus L}$.
	\end{enumerate}
\end{definition}

We require two key results from \cite{brady_branched_1999}.

\begin{lem}[Brady \cite{brady_branched_1999}, Lemma 5.3]
	There is a natural surjection $\widehat{K}\to K$ and $\widehat{K}$ is a piecewise Euclidean cube complex. 
\end{lem}

\begin{lem}[Brady \cite{brady_branched_1999}, Lemma 5.5]\label{npcbranch}
	If $L$ is a finite graph, then $\widehat{K}$ is non-positively curved. 
\end{lem}

\subsection{Bestvina--Brady Morse theory}

While Bestvina--Brady Morse theory is defined in the more general setting of affine cell complexes, we shall restrict to the case of CAT(0) cube complexes. 

For the remainder of this section, let $X$ be a CAT(0) cube complex and let $G$ be a group which acts freely, cellularly, properly and cocompactly on $X$. Let $\phi\colon G\to\ZZ$ be a homomorphism and let $\ZZ$ act on $\RR$ by translations.

Let $\chi_c$ be the characteristic map of the cube $c$.

\begin{definition}
	We say that a function $f\colon X\to \RR$ is a {\em $\phi$-equivariant Morse function} if it satisfies the following 3 conditions.
	\begin{itemize}
		\item For every cube $c\subseteq X$ of dimension $n$, the map $f\chi_c\colon [0,1]^n\to\RR$ extends to an affine map $\RR^n\to\RR$ and $f\chi_c\colon [0,1]^n\to\RR$ is constant if and only if $n=0$.
		\item The image of the $0$-skeleton of $X$ is discrete in $\RR$.
		\item $f$ is $\phi$-equivariant, i.e. $f(g\cdot x) = \phi(g)\cdot f(x)$.
	\end{itemize}
\end{definition}

\begin{definition}
	For a non-empty closed subset $I\subseteq\RR$ we denote by $X_I$ the preimage of $I$. The sets $X_I$ are known as {\em level sets}. Also for any real number $t$ we simply write $X_t$ for $X_{\{t\}}$.
\end{definition}

The kernel $H$ of $\phi$ acts on the cube complex $X$ preserving each level set $X_{I}$. The topological properties of the level sets allow us to gain information about the finiteness properties of the kernel. We need to examine the topology of the level sets and how they vary as we pass to larger level sets. 

\begin{thm} [Bestvina--Brady, \cite{bestvina_morse_1997}, Lemma 2.3]
	If $I\subseteq I'\subseteq\RR$ are connected and $X_{I'}\smallsetminus X_{I}$ contains no vertices of $X$, then the inclusion $X_I\hookrightarrow X_{I'}$ is a homotopy equivalence. 
\end{thm}

If $X_{I'}\smallsetminus X_{I}$ contains vertices of $X$, then the topological properties of $X_{I'}$ can be very different from those of $X_I$. This difference is encoded in the ascending and descending links. 

\begin{definition}
	The {\em ascending link} of a vertex is 
	
	$\lk_{\uparrow}(v,X) = \bigcup \{{\rm Lk}(w,c)\mid\chi_c(w) = v$ and $w$ is a minimum of $f\chi_c\}\subseteq {\rm Lk}(v, X).$
	
	The {\em descending link} of a vertex is 
	
	${\rm Lk}_{\downarrow}(v,X) = \bigcup \{{\rm Lk}(w,c)\mid\chi_c(w) = v$ and $w$ is a maximum of $f\chi_c\}\subseteq {\rm Lk}(v, X).$
\end{definition}

\begin{thm} [Bestvina--Brady, \cite{bestvina_morse_1997}, Lemma 2.5]\label{homoeq}
	Let $f$ be a Morse function. Suppose that $I\subseteq I'\subseteq\RR$ are connected and closed with $\min I = \min I'$ (resp. $\max I = \max I')$, and assume $I'\smallsetminus I$ contains only one point $r$ of $f\big(X^{(0)}\big)$. Then $X_{I'}$ is homotopy equivalent to the space obtained from $X_I$ by coning off the descending (resp. ascending) links of $v$ for each $v\in f^{-1}(r)$. 
\end{thm}

We can now deduce a lot about the topology of the level sets. We know how they change as we pass to larger intervals and so we have the following. 

\begin{cor}[Bestvina--Brady, \cite{bestvina_morse_1997}, Corollary 2.6]\label{cor1} Let $I,I'$ be as above.
	\begin{enumerate}
		\item If each ascending and descending link is homologically $(n-1)$-connected, then the inclusion $X_I\hookrightarrow X_{I'}$ induces an isomorphism on $H_i$ for $i\leq n-1$ and is surjective for $i=n$.
		\item If the ascending and descending links are connected, then the inclusion $X_I\hookrightarrow X_{I'}$ induces a surjection on $\pi_1$. 
		\item If the ascending and descending links are simply connected, then the inclusion $X_I\hookrightarrow X_{I'}$ induces an isomorphism on $\pi_1$.
	\end{enumerate}
\end{cor}

Knowing that the direct limit of this system is a contractible space allows us to compute the finiteness properties of the kernel of $\phi$. 

\begin{theorem}[Bestvina--Brady, \cite{bestvina_morse_1997}, Theorem 4.1]\label{bbmorse}
	Let $f\colon X\to \RR$ be a $\phi$-equivariant Morse function and let $H=\ker(\phi)$. If all ascending and descending links are simply connected, then $H$ is finitely presented (i.e. is of type $F_2$).
\end{theorem}

We would also like to have conditions which allow us to deduce that $H$ does not satisfy certain other finiteness properties. The following is a rephrasing of the argument in \cite{bux_bestvina-brady_1999other}.

\begin{lem}\label{lem:retractlinks}
	Let $f$ be a $\phi$-equivariant Morse function. Assume  all ascending and descending links are connected. Assume further that there is a vertex $v$ such that the $\lk(v, X)$ retracts onto $\lk_{\uparrow}(v, X)$ and $\pi_1(\lk_{\uparrow}(v, X))\neq 0$. Then $H = \ker(\phi)$ is not finitely presented. 
\end{lem}
\begin{proof}
	Assume that $H$ is finitely presented. By Corollary \ref{cor1}, the level set $X_0$ is connected, so there are finitely many $H$ orbits of loops which generate $\pi_1(X_0)$. We can find an interval $I$ such that each of these loops is trivial in $X_I$. Once again by Corollary \ref{cor1} we see that that the inclusion $X_0\to X_I$ induces a surjection on fundamental groups. However, this map is also trivial. Thus $X_I$ is simply connected. 
	
	Let $l$ be such that $I\subseteq [-l, l]$. Let $k_0$ be the height of the vertex $v$ form the hypothesis. Using the action of $\ZZ$ on $X$, we can find $g\in G$ such that $f(g\cdot v) > l$. 
	Let $w = g\cdot v$. Then there is a $k > l$ such that $f(w) = k$ and $\pi_1(\lk_{\downarrow}(w, X))\neq 0$.
	Let $L = [-l, k-\epsilon]$, since this contains $J$ and the map $\pi_1(X_J)\to \pi_1(X_L)$ is surjective, we see that $\pi_1(X_L) = 0$. There is a retraction of $X\smallsetminus\{w\}$ onto $\lk(w, X)$ which further retracts onto the space $\lk_{\downarrow}(w, X)$. Restricting this retraction to $X_L\subseteq X\smallsetminus\{w\}$ we get a retraction from $X_L\to \lk_{\downarrow}(w, X)$. This gives a surjection $\pi_1(X_L)\to \pi_1(\lk_{\downarrow}(w, X))$. However $\pi_1(\lk_{\downarrow}(w, X))$ is non trivial by assumption giving the required contradiction. 
\end{proof}

\section{Main Construction}

We detail a construction of several new hyperbolic groups which have subgroups with interesting finiteness properties. These examples come in two flavours; we can construct groups of type $F_2$ not $F_3$ giving more examples similar to those in \cite{brady_branched_1999, lodha_hyperbolic_2017other,kropholler_finitelypresented_2018}. We also give the first construction of groups of type $FP_2$ which are not finitely presented that are contained in hyperbolic groups. 

We only detail the construction of groups of type $FP_2$ not $F_2$. To create groups which are of type $F_2$ not $F_3$ one should construct $\Gamma_B$ by starting with a simply connected complex $L$ that has no local cut points. One also requires the stated strengthening of Lemma \ref{lem:homologyofdouble}. The proof then runs in exactly the same way however $S(L')$ will be simply connected throughout. 

We build the hyperbolic groups by taking appropriate branched covers of $\xcplx$. As such we must begin by describing the flag complexes $\Gamma_A$ and $\Gamma_B$. 

Let $V_n$ be a discrete set with $n$ points. Set $\Gamma_A = V_4\ast V_4\ast V_4$ which is a join of discrete sets. 

\begin{rem}
	We take $V_4$ for concreteness and this construction works for $V_n$ as long as $n\geq 4$. We shall not give the proof in this general case as it obfuscates some details. 
\end{rem}

The complex $\Gamma_B$ is a little more delicate and the procedure for obtaining it is described below. 

\begin{definition}
	For a simplicial complex $L$ the {\em octahedralisation} $S(L)$ is defined as follows. For each vertex $v$ of $L$, let $S^0_v = \{v^+, v^-\}$ be a copy of $S^0$. For every simplex $\sigma$ of $L$ take $S_{\sigma} = \ast_{v\in \sigma}S^0_v$. If $\tau<\sigma$, then there is a natural map $S_{\tau}\to S_{\sigma}$. $S(L) = \bigcup_{\sigma< L} S_{\sigma}/\sim$, where the equivalence relations $\sim$ is generated by the inclusions $S_{\tau}\to S_{\sigma}.$
\end{definition}

\begin{rem}
	The map defined by $S^0_v\to \{v\}$ extends to a retraction of $S(L)$ to $L$ (not a deformation retraction). 
In particular, if $L$ is connected and $\pi_1(L)\neq 0$, then $\pi_1(S(L))\neq 0$. 
\end{rem}

It is proved in \cite{bestvina_morse_1997} that if $L$ is a flag complex, then $S(L)$ is a flag complex. 

\begin{definition}
	A connected simplicial complex $L$ has {\em no local cut points}, abbreviated to nlcp, if $\lk(v, L)$ is connected and not equal to a point for all vertices $v\in L$.
\end{definition}

The following lemma can be found in \cite{kropholler_almost_2018, leary_uncountably_2015other}.

\begin{lem}\label{lem:homologyofdouble}
	Assume $L$ has nlcp. If $H_1(L) = 0$, then $H_1(S(L)) = 0$. 
\end{lem}

\begin{rem}
	For the $F_2$ not $F_3$ case of the main theorem a stronger version of this theorem is required. Namely, using the argument proving Lemma 6.2 of \cite{kropholler_almost_2018}, one can show that in the above case if $\pi_1(L) = 0$, then $\pi_1(S(L)) = 0$. 
\end{rem}

Given a connected simplicial complex $L$, there is a homotopy equivalent complex $K$ with no local cut points. One way of obtaining such a complex is to take the mapping cylinder of the inclusion of the 1-skeleton. For details see \cite{leary_uncountably_2015other}.

\begin{rem}
	Given a tripartite complex $L$, there is a natural tripartite structure on $S(L)$. 
\end{rem}

We are now ready to define the tripartite complex $\Gamma_B$. Given a simplicial complex $L$, let $L'$ denote the barycentric subdivision of $L$. 

Let $L$ be a 2-dimensional simplicial complex with no local cut points. Let $\Gamma_B = S(S(L'))$. Note that $L'$ has a natural tripartite structure, that is $L'\subseteq B_1\ast B_2\ast B_3$ where $B_i$ is the set of vertices that are barycenters of $(i-1)$-cells of $L$. Also given a tripartite structure on a complex $K\subseteq K_1\ast K_2\ast K_3$, there is a natural tripartite structure on $S(K)$ given by viewing $S(K)$ as a subspace of $S(K_1)\ast S(K_2)\ast S(K_3)$.

Before moving on it will be useful to understand the links of vertices in $\xcplx$. 

The vertices in $\xcplx$ are of the form $(v_1, v_2, v_3)$ where $v_i\in A_i\cup B_i$. Let $a_i\in A_i$ and $b_i\in B_i$. We will use Lemma \ref{lem:links.in.xcplx} to compute $\lk((v_1, v_2, v_3), \xcplx)$. 

The first case is that of a vertex of the form $(a_1, a_2, a_3)$ we obtain the following, 
$$\lk((a_1, a_2, a_3), \xcplx)=\lk(\emptyset, \Gamma_B)\ast \lk([a_1, a_2, a_3], \Gamma_A) = \Gamma_B\ast \emptyset = \Gamma_B.$$ 
Similarly $\lk((b_1, b_2, b_3), \xcplx) = \Gamma_A$. 

We now examine the intermediate vertices. For vertex of the form $(a_1, b_2, b_3)$ we have the following $$\lk((a_1, b_2, b_3), \xcplx) = \lk(a_1, \Gamma_A)\ast \lk([b_2, b_3], \Gamma_B).$$ By retracting onto $L'$ we can see that $\lk([b_2, b_3], \Gamma_B) = S(S(\lk([\overline{b_2}, \overline{b_3}], L')))$. Since $[b_2, b_3]$ is an edge joining the barycenter of a 1-cell to the barycenter of a 2-cell, we see that $\lk([\overline{b_2}, \overline{b_3}], L') = S^0$. Also $\lk(a_1, \Gamma_A) = V_4\ast V_4$. Thus, $\lk((a_1, b_2, b_3), \xcplx) = V_4\ast V_4\ast S(S(S^0))$. 

Making a similar analysis we obtain that $$\lk((b_1, a_2, b_3), \xcplx) = V_4\ast V_4\ast S(S(S^0)).$$ 
Considering a vertex of the form $(b_1, b_2, a_3)$, we have $$\lk((b_1, b_2, a_3), \xcplx) = \lk(a_3, \Gamma_A)\ast \lk([b_1, b_2], \Gamma_B).$$ We can still retract to show that $\lk([b_1, b_2], \Gamma_B) = S(S(\lk([\overline{b_1}, \overline{b_2}], L')))$, the latter is a discrete set and is non-empty since $L'$ has nlcp. Thus, $$\lk((b_1, b_2, a_3), \xcplx) = V_4\ast V_4\ast S(S(V_n)),$$ where the $n$ depends on the vertices $b_1$ and $b_2$. 

For a vertex of the form $(a_1, a_2, b_3)$ we obtain that $$\lk((a_1, a_2, b_3), \xcplx) = V_4\ast S(S(\lk(\overline{b_3}, L'))).$$ Since $\overline{b_3}$ is the barycenter of a 2-cell, we obtain that $\lk(\overline{b_3}, L') = \hexagon$ where $\hexagon$ is a copy of $S^1$ triangulated with 6 vertices. Thus, $$\lk((a_1, a_2, b_3), \xcplx) = V_4\ast S(S(\hexagon)).$$

For a vertex of the form $(a_1, b_2, a_3)$ we similarly obtain that $$\lk((a_1, b_2, a_3), \xcplx) = V_4\ast S(S(\lk(\overline{b_2}, L'))).$$ In this case $\overline{b_2}$ is the barycenter of a 1-cell and thus $\lk(\overline{b_2}, L') = V_n\ast S^0$. Since $L'$ has nlcp we see that $n\geq 1$ and depends on the chosen vertex $b_2$.

Finally a vertex of the form  $(b_1, a_2, a_3)$ has link $$\lk((b_1, b_2, b_3), \xcplx) = V_4\ast S(S(\lk(\overline{b_1}, L'))).$$ The vertex $\overline{b_1}$ is a vertex of the original complex $L$ and thus $\lk(\overline{b_1}, L') = \lk(\overline{b_1}, L)'$. Set $\Lambda = \lk(\overline{b_1}, L)$. Thus $$\lk((b_1, b_2, b_3), \xcplx) = V_4\ast S(S(\Lambda'))).$$

These links are summarised in table \ref{table:links}.

\begin{table}[h!]
	\centering
	\begin{tabular}{|c|rcl|}
		\hline
		Vertex of $\xcplx$ &\multicolumn{3}{c|}{Link of the vertex in $\xcplx$}\\
		\hline
		$(a_1, a_2, a_3)$& $\Gamma_B $&$=$&$ S(S(L'))$ \\
		$(a_1, a_2, b_3)$& $\lk(b_3, \Gamma_B)\ast V_4 $&$=$&$ S(S(\hexagon))\ast V_4$\\
		$(a_1, b_2, a_3)$& $\lk(b_2, \Gamma_B)\ast V_4 $&$=$&$ S(S(V_n\ast S^0))\ast V_4$\\
		$(a_1, b_2, b_3)$& $\lk([b_2, b_3], \Gamma_B)\ast V_4\ast V_4 $&$=$&$ S(S(S^0))\ast V_4\ast V_4$\\
		$(b_1, a_2, a_3)$& $\lk(b_1, \Gamma_B)\ast V_4 $&$=$&$ S(S(\Lambda'))\ast V_4$\\
		$(b_1, a_2, b_3)$& $\lk([b_1, b_3], \Gamma_B)\ast V_4\ast V_4 $&$=$&$ S(S(S^0))\ast V_4\ast V_4$\\
		$(b_1, b_2, a_3)$& $\lk([b_1, b_2], \Gamma_B)\ast V_4\ast V_4 $&$=$&$ S(S(V_n))\ast V_4\ast V_4$\\
		$(b_1, b_2, b_3)$& $\Gamma_A $&$=$&$ V_4\ast V_4\ast V_4$\\
		\hline
	\end{tabular}
	\caption[The links of various vertices in $\xcplx$.]{The links of various vertices in $\xcplx$. Here $\Lambda'$ is the barycentric subdivision of the link of the vertex corresponding to $b_1$ in $L$. Also, \hexagon is a copy of $S^1$ triangulated with six vertices. It should be noted that, since $L$ has nlcp, the graph $\Lambda$ is connected and $V_n$ contains at least one point. }
	\label{table:links}
\end{table}

\subsection{The Morse theory}

To define a Morse function on the complex $\xcplx$ we begin by defining a Morse function on each of the graphs $A_i\ast B_i$. We do this by putting an orientation on each edge. 

Given $\Gamma_A, \Gamma_B$ as above consider the cube complex $\xcplx$. This is naturally contained in $X = \prod_{i=1}^{3} A_i\ast B_i$. 

There is a bipartion of the vertex set of $S(\Gamma)$ for any complex $\Gamma$ extending the bipartition of $S^0$. Thus we can bipartition the sets $B_i$ from the final application of octahedralisation. Take any bipartition of $A_i = A_i^+\sqcup A_i^-$ such that $|A_i^+| = 2$. 

Let $v\in A_i$ be a vertex and $w\in B_i$ be a vertex. 

\begin{itemize}
	\item Orient the edge from $v$ to $w$ if $v\in A_i^+$ and $w\in B_i^+$ or if $v\in A_i^-$ and $w\in B_i^-$. 
	\item Orient the edge from $w$ to $v$ if $v\in A_i^+$ and $w\in B_i^-$ or if $v\in A_i^-$ and $w\in B_i^+$. 
\end{itemize}

We can then use this to define a Morse function on $X$ and restricting this we get a Morse function $f$ on $\xcplx$. 

The ascending (resp. descending) link is the full subcomplex of the link spanned by vertices corresponding to those edges oriented toward (resp. away from) the vertex.  Let $a_1\in A_i$ and $b_i\in B_i$. An edge emananting from  the vertex $(v_1, v_2, v_3)$ correspond to changing a coordinate of the form $a_i$ to a coordiante of the form $b_i$ or vice versa.

In the case that the coordinate $a_i$ was in $A_i^+$, the outgoing edges correspond to $B_i^+$, similarly if $a_i\in A_i^-$, then the outgoing edges correspond to $B_i^-$. 

The ascending and descending links are the full subcomplexes of the link spanned by vertices corresponding to edges oriented towards or away from the vertex respectively. The ascending and descending links for this Morse function are in Table \ref{table:ascdesc}.

\begin{table}[h!]
	\centering
	\begin{tabular}{|c|c|}
		\hline
		Vertex & Ascending Link/Descending link\\
		\hline
		$(a_1, a_2, a_3)$& $S(L')$ \\
		$(a_1, a_2, b_3)$& $S(\hexagon)\ast S^0$\\
		$(a_1, b_2, a_3)$& $S(V_n\ast S^0)\ast S^0$\\
		$(a_1, b_2, b_3)$& $S(S^0)\ast S^0\ast S^0$\\
		$(b_1, a_2, a_3)$& $S(\Lambda')\ast S^0$\\
		$(b_1, a_2, b_3)$& $S(S^0)\ast S^0\ast S^0$\\
		$(b_1, b_2, a_3)$& $S(V_n)\ast S^0\ast S^0$\\
		$(b_1, b_2, b_3)$& $S^0\ast S^0\ast S^0$\\
		\hline
	\end{tabular}
	\caption{Ascending and descending links for the Morse function $f$. }
	\label{table:ascdesc}
\end{table}

We can see that the ascending and descending links of $f$ are simply connected with the exception of $S(L')$, cf. Table \ref{table:ascdesc}. 

\subsection{The branched cover} \label{sec:branched}

We are now ready to take the branched cover of $\xcplx$. Consider the subcomplex $Z$ of $\prod_{i=1}^3 A_i\ast B_i$
\begin{align*}
Z =& (A_1\ast B_1)\times B_2\times A_3\, \sqcup\\
&A_1\times (A_2\ast B_2)\times B_3 \,\sqcup\\
&B_1\times A_2\times (A_3\ast B_3).
\end{align*}

We take as our branching locus $Y = Z\bigcap \xcplx$. We now check that this satisfies \cref{def:branchinglocus}. 

Since the link of each vertex in $\xcplx$ is tripartite, there is an embedding $\lk(v, \xcplx)\to C_1\ast C_2\ast C_3$. Considering each of the complexes with the CAT(1) metric this map is distance non-increasing. The vertices corresponding to the branching locus $Y$ form one of the sets $C_i$. The distance between any two point in $C_i$ is $\pi$ in $C_1\ast C_2\ast C_3$. Thus the distance is at least $\pi$ in $\lk(v, \xcplx)$. Thus the subcomplex $Y$ is locally isometrically embedded.

We must now check that $\lk(c, \xcplx)\smallsetminus\lk(c, Y)$ is non-empty and connected for all cubes $c\in Y$. 
Since $\lk(e, Y) = \emptyset$ for all edges $e$ we see that we only need to check the vertices of $Y$. 

Let $v$ be a vertex of $Y$. We will examine the case that $v = (a_1, a_2, b_3)$, the other cases are similar. Then $\lk(v, \xcplx)$ is a join $\lk([a_1, a_2], \Gamma_A)\ast \lk(b_3, \Gamma_B)$. Since $\Gamma_B$ has no local cut points we see that $\lk(b_3, \Gamma_B)$ is a connected graph. Thus $\lk([a_1, a_2], \Gamma_A)\ast \lk(b_3, \Gamma_B)$ has no local cut points and removing a discrete set from a space with no local cut points preserves connectivity. 

Thus $Y$ is a branching locus in the sense of \cref{def:branchinglocus}. 

To take a  branched cover we use a similar procedure to that detailed in \cite{brady_branched_1999,kropholler_almost_2018}. 

Define $\Gamma_{A_{ij}}$ to be the subcomplex of $\Gamma_A$ obtained by taking the full subcomplex on the vertices in $A_i\cup A_j$. Define $\Gamma_{B_{ij}}$ similarly. Since $\Gamma_A$ and $\Gamma_B$ have no local cut points both of these are connected bipartite graphs. 

We project $\xcplx\smallsetminus Y$ to the 2-dimensional complexes $\xcplxij{i}{j}\smallsetminus (A_i\times B_j)$ for $(i, j)\in \{(3, 2), (2, 1), (1, 3)\}$. 
These are restrictions of the maps $\rho_k\colon\prod_{l=1}^3 A_l\ast B_l\to (A_i\ast B_i)\times (A_j\ast B_j)$ projecting away from the $k$-th coordinate. 
Since we have removed the set $A_i\times B_j\times (A_k\ast B_k)$ we can see that there are no vertices in the image of the form $(a, b)$ for $a\in A_i, b\in B_j$. 

Each of the complexes $\xcplxij{i}{j}\smallsetminus (A_i\times B_j)$ deformation retracts onto a graph $\Lambda_{ij}$. Indeed each square of $\xcplxij{i}{j}$ has exactly one corner in $A_i\times B_j$ and we can homotope radially from this corner. Thus $\Lambda_{ij}$ is a subgraph of $\Xi_{ij}$, where $\Xi_{ij}$ is the subgraph of $(A_i\ast B_i)\times (A_j\ast B_j)$ spanned by vertices not in $A_i\times B_j$ i.e.  $\Xi_{ij} = \big((A_i\ast B_i) \times A_j\big)\cup \big(B_i\times (A_j\ast B_j)\big).$

We will take a cover by giving a representation of the fundamental group to $S_p$, where $S_p$ is the symmetric group on $p$ letters. 
We briefly outline the key points that this cover will satisfy before delving into the details.

\begin{table}[h!]
	\centering
	\begin{tabular}{|c|c|c|c|}
		\hline
		Vertex & Link & Link container & Branch set\\
		\hline
		$(a_1, a_2, b_3)$&$S(S(\boldhexagon))\ast V_4$   &$B_1\ast B_2\ast A_3$&$B_2$\\
		$(a_1, b_2, a_3)$&$\bm{S(S(V_n))}\ast V_4 \ast S(S(S^0))$&$B_1\ast A_2\ast B_3$&$B_1 = S(S(V_n))$\\
		$(a_1, b_2, b_3)$&$S(S(S^0))\ast \bm{V_4}\ast V_4$&$B_1\ast A_2\ast A_3$&$A_2 = V_4$\\
		$(b_1, a_2, a_3)$&$S(S(\Lambda{\bm{'}}))\ast V_4$   &$A_1\ast B_2\ast B_3$&$B_3$\\
		$(b_1, a_2, b_3)$&$V_4\ast S(S(S^0))\ast \bm{V_4}$&$A_1\ast B_2\ast A_3$&$A_3 = V_4$\\
		$(b_1, b_2, a_3)$&$\bm{V_4}\ast V_4\ast S(S(V_n))$&$A_1\ast A_2\ast B_3$&$A_1 = S(S(V_n))$\\
		\hline
	\end{tabular}
	\caption{Links in $\xcplx$ and their branch sets. The link container is the natural tripartite complex which contains the link. The branch sets are in bold in the link. In the first row the branch set is $S(S(W))$ where $W$ is the three bold vertices of the hexagon. In the fourth row the branch set is $S(S(B))$ where $B$ are the barycenters of edges in $\Lambda'$.}
	\label{table:branchlinks}
\end{table} 

\begin{lem}\label{lem:defretractlinks}
	Let $v$ be a vertex of $Y$. 
	The space $\lk(v, \xcplx)\smallsetminus\lk(v, Y)$ deformation retracts onto a complete bipartite graph $\Omega_v$. 
\end{lem}
\begin{proof}
	We have already seen that $\lk(v, \xcplx)$ is the join of a bipartite graph and a discrete set. 
	We will carefully examine the case that $v = (a_1, a_2, b_3)$, the analysis for the other cases is the same. 
	
	From \Cref{table:branchlinks}, we see that $\lk(v, \xcplx)\subseteq B_1\ast B_2\ast A_3$ and $\lk(v, Y)  \subseteq B_2$.
	By the definition of $Y$, we see that $B_2\cap \lk(v, \xcplx) = \lk(v, Y)$. 
	Let $C_2 = \lk(v, Y)$. 
	We can see that $\lk(v, \xcplx)\subseteq B_1\ast C_2\ast A_3$. 
	
	There is a deformation retraction $(B_1\ast C_2\ast A_3)\smallsetminus C_2\to B_1\ast A_3$.
	To see this note that each triangle of $B_1\ast C_2\ast A_3$ has a single vertex in $C_2$. 
	We can push radially out from this vertex to achieve the desired deformation retraction. 
	
	For any subcomplex $\Sigma$ of $B_1\ast C_2\ast A_3$ we still have that each triangle has a single vertex in $C_2$. 
	Thus we can push out radially in each triangle from this vertex and gain a deformation retraction $\Sigma\smallsetminus C_2\to \Sigma\cap B_1\ast A_3$. 
	
	In the case of $v$, we see, from \Cref{table:links}, that $\lk(v, \xcplx) = S(S(\hexagon))\ast V_4$, where $\hexagon$ is a copy of $S^1$ triangulated with 6 vertices. 
	Thus, in this case we have that $\lk(v, X)\cap B_1\ast A_3$ is $S(S(W))\ast V_4$, where $W$ is a set of three vertices corresponding to alternating vertices of $\hexagon$. 
	This is a complete bipartite graph, which we label $\Omega_v$. 
	
	Examining \Cref{table:branchlinks}, shows that the same result holds for all other vertices of the branching locus. 
\end{proof}

When taking our cover we will be particularly interested in loops of length 4 in $\Omega_v$. 
These correspond to commutators in the fundamental group of $\xcplx\smallsetminus Y$. 

These loops play a large role in both proving \cref{thm:morsebranch} and \cref{thm:hypbranch}. 
Let $\bar{v}$ be a vertex in the preimage of $v$. 
In the link of this vertex there is a cover $\overline{\Omega}_v$ of $\Omega_v$, we will require the following two properties. 
\begin{itemize}
	\item $\overline{\Omega}_v$ has no loops of length 4. 
	\item For each loop of length 4 in $\Omega_v$ its preimage is connected in $\overline{\Omega}_v$. 
\end{itemize}
Under the map of $\xcplx\smallsetminus Y\to \Lambda_{ij}$, these loops of length 4 map to certain loops of length 8. 
We will be particularly intereseted in the image of these loops of length 8 under the representation. 
With these properties in mind, we give a formal definition of the cover used.

Let $q_{ij}$ be a prime to be determined later. Let $\alpha$ be a permutation of order $q_{ij}$ and $\beta$ be a permutation such that $\beta\alpha{\beta^{-1}} = \alpha^l$ where $l$ is a generator of $\ZZ_{q_{ij}}^{\times}$. Note that $[\alpha^a, \beta^b] = \alpha^{a(l^b-1)}$, this is a non-trivial power of $\alpha$ if $0<a<q_{ij}$ and $0<b<q_{ij}-1$ and is hence an element of order $q_{ij}$.

We define a homomorphism $\phi_{ij}\colon \pi_1(\Lambda_{ij})\to S_{q_{ij}}$ by labeling the graph $\Xi_{ij}$ with elements of $S_{q_{ij}}$. We label the edges of $(A_i\ast B_i) \times B_j$ by powers of $\alpha$ such that no two edges oriented towards a vertex have the same label. We label the edges of $A_i\times (A_j\ast B_j)$ by powers of $\beta$ with the same condition. This required us to pick $q_{ij}$ larger than the valence of any vertex in $A_j\ast B_j$ or $A_i\ast B_i$. 

This gives us three representations of $\pi_1(\xcplx\smallsetminus Y)$ in the symmetric groups $S_{q_{ij}}$ for $(i, j)\in \{(1, 2), (2, 3), (3, 1)\}$. We can combine these to get a representation into $S_{q_{12}q_{23}q_{31}}$ and take the cover corresponding to the stabiliser of 1 in this subgroup. We then lift the metric and complete to obtain the branched cover $\xbarcplx$. This is a non-positively curved cube complex by Lemma \ref{npcbranch}. 

\subsubsection{The links in the branched cover}

Throughout, we will consider the link of a vertex $v$ to be the $\frac{1}{4}$-sphere centered at $v$. 
Shortly, we will look at the ascending and descending links for a new Morse function. We first compute the link of each vertex in $\xbarcplx$. The link of a vertex in $\xbarcplx$ is a branched cover of a link in $\xcplx$. 

Recall, there are 8 types of vertex in $\xcplx$ we will examine each type in turn. The vertices which are of the form $(a_1, a_2, a_3)$ or $(b_1, b_2, b_3)$ are disjoint from the branching locus, thus the link of such a vertex lifts to the cover $\xbarcplx$. 

Let $v$ be a vertex on the branching locus $Y$. Let $w$ be a vertex in $\xbarcplx$ mapping to $v$. We get a branched covering $\lk(w, \xbarcplx)\to \lk(v, \xcplx)$. The points at which this map is not a local isometry are the vertices corresponding to $Y$ in $\lk(v, \xcplx)$. We call these vertices the {\em branch set}. We will denote $\lk(v, \xcplx)\smallsetminus Y$ by $\overline{\lk}(v, \xcplx)$

To understand the covering of each link we look at the three projections. We will examine the case of the vertex $v = (a_1, b_2, a_3)$. Under the projection $\rho_2$ this vertex is mapped to $(a_1, a_3)\in \xcplxij{1}{3}\smallsetminus A_1\times B_3$. Thus the link of $v$ is mapped into the $\frac{1}{4}$-neighbourhood of $(a_1, a_3)$. This neighbourhood is contractible. Thus the composition $$\pi_1(\overline{\lk}(v, \xcplx))\to \pi_1(\xcplx\smallsetminus Y)\to \pi_1(\xcplxij{1}{3}\smallsetminus (A_1\times B_3))$$ is trivial. Similarly for $\rho_3$. Thus neither of these representations change $\overline{\lk}(v, \xcplx)$. 

By \Cref{lem:defretractlinks}, there is a deformation retraction from $\overline{\lk}(v, \xcplx)$ to the subgraph $\Omega_v$ spanned by the vertices not in the branch set. This is a join of two discrete sets. 

In the case of $v = (a_1, b_2, a_3)$, we see that $\rho_1$ gives an isometric inclusion $\Omega_v$  into $\lk((b_2, a_3), \xcplxij{2}{3})$. 
Further under the composition $$\Omega_v\to \xcplx\smallsetminus Y\to \xcplxij{2}{3}\smallsetminus (A_2\times B_3)\to  \Xi_{23}$$ each loop $\omega$ of length 4 in $\Omega_v$ is sent to a loop $\xi$ of length 8 in $\Xi_{23}$.

Under the representation to $S_{q_{23}}$ the element corresponding to $\gamma$ is sent to a non-trivial power of $\alpha$. Thus this is a $q_{23}$-cycle and under the covering the preimage of $\omega$ is connected and is a loop of length $4q_{23}$. 

Each vertex in the branch set cones off a subgraph $\Omega'$ of $\Omega_v$. When we complete the cover, each vertex in the completion of the link will cone off the corresponding lift of $\Omega'$.

\begin{rem}\label{rem:4loops}
	Since all the loops of length 4 have connected preimage the cover of $\Omega_v$ will have no loops of length 4.
\end{rem}

This will be useful in proving the hyperbolicity of the branched cover shortly. 

There is a naturally defined Morse function $\bar{f}$ on $\xbarcplx$ by composing the Morse function on $\xcplx$ with the branched covering map. 
The ascending and descending links of $\bar{f}$ are the preimages of the ascending and descending links of $f$ under the branched covering map. 

We now examine these ascending and descending links. We can see the ascending and descending links in $\xbarcplx$ as branched covers of the ascending and descending links in $\xcplx$ over the branch set. Since a neighbourhood of any point not on the branching locus lifts to the branched cover, the ascending and descending links of vertices of the form $(a_1, a_2, a_3)$ and $(b_1, b_2, b_3)$ remain unchanged. 

The other links change under this branched covering map. We must show that these branched covers are simply connected. For what follows the branch set will be referred to as $V$. See \cref{table:branchlinks} for a description of the branch set in each vertex.

\begin{lem}\label{lem:int4cycles}
	In all the complexes $\Gamma$ in the Table \ref{table:ascdesc} there is an ordering $v_1, v_2, \dots$ on the set $V$ such that $\lk(v_i, \Gamma)\cap \bigcup_{j<i} \st(v_j, \Gamma)$ is connected and covered by loops of length 4.
	
	Furthermore, $\lk(v, \xcplx)$ is covered by the stars of links in the branch set. 
\end{lem}
\begin{proof}
	For the cases of vertices of type $(b_1, a_2, b_3), (b_1, b_2, a_3), (a_1, b_2, a_3)$ and $(a_1, b_2, b_3)$ this is clear since $\lk(v_i)\cap \st(v_j) = \lk(v_i)$ and $\lk(v_i)$ is the suspension of a discrete non-empty set with at least 2 elements.  
	
	We provide the proof for the complex $S(\Lambda{'})\ast S^0$, the proof for $S(\hexagon)\ast S^0$ being the special case that $\Lambda$ is a triangle. 
	
	The key element of the proof is that given a connected bipartite graph there is an ordering on the vertices in one part such that the star of vertex intersects the star of a previous vertex in a non-empty set. 
	
	Let $W = \{w_1, \dots w_n\}$ be the vertices in $\Lambda'$ coming from the barycenters of edges in $\Lambda$. Assume that under this ordering the star of $w_i$ intersects the star of $w_j$ for some $j<i$. The set $V$ is the octahedralisation of $W$. We claim that the ordering $w_1^-, w_1^+, w_2^-, w_2^+, \dots, w_n^+$ is the desired ordering. 
	
	The link of each of the vertices is non-empty. Also the star of each vertex intersects the star of a previous vertex. Since we are taking the octahedralisation then we can see that $\lk(v_i, \Gamma)\cap \bigcup_{j<i} \st(v_j, \Gamma)$ is the suspension of a discrete set. This set contains at least two elements so will be connected and can be covered by loops of length 4 as desired. 
	
	The second statement follows since the branch is one part of the tripartite structure and the complex has no local cut points 
\end{proof}

We are now ready to prove that the ascending and descending links of these vertices will still be simply connected. 

\begin{thm}\label{thm:morsebranch}
	Let $\xbarcplx$ be the cube complex constructed above as a branched cover of $\xcplx$. Then there is an $S^1$ valued Morse function $f$ on ${\xbarcplx}$ such that the ascending and descending links are simply connected or $S(L')$. 
\end{thm}
\begin{proof}
	We need only check that we preserve simple connectedness in the ascending and descending links of vertices in $Y$. 
	
	Let $v$ be a vertex of $Y$. In $\lk(v, \xcplx)$ we are removing the branch set and taking a cover. 
	For each vertex $w$ in the branch set $\lk(w, \lk(v, \xcplx))$ is a graph which can be covered by loops of length 4. 
	We can order these loops of length 4 such that each successive loop has a non-trivial intersection with on of the previous loops. 
	The preimage of each of the loops of length 4 is a single loop. 
	Thus the preimage of $\lk(w, \lk(v, \xcplx))$ is a connected graph. 
	When we replace the preimages of $w$ we cone of this graph. 
	
	With the ordering of the vertices from \cref{lem:int4cycles} we are gluing a sequence of contractible sets along connected subspaces. The result of this procedure is simply connected. Thus the ascending and descending links are simply connected for any vertex in the branching locus. 
	
	For vertices not on the branching locus, the ascending and descending links remain unchanged. Thus they are all either simply connected or $S(L')$. 
\end{proof}

\begin{cor}\label{cor:fp2notf2}
	Let $L$ be a simplicial complex with nlcp whose fundamental group is perfect. Let $\xbarcplx$ be the cube complex constructed above with $\Gamma_B = S(S(L'))$. Then the kernel of $f_*$ is of type $FP_2$ but not finitely presented. 
\end{cor}
\begin{proof}
	By Lemma \ref{lem:retractlinks} that the finiteness properties of the kernel are controlled by the homotopy of $S(L')$. Thus if $L'$ has non-trivial perfect fundamental group, then the kernel of $f_*$ is of type $FP_2$ but not finitely presented.
\end{proof}

\begin{rem}
	If $L$ is simply connected, then the above kernel is finitely presented but not of type $FP_3$. 
\end{rem}

\subsection{The cover is hyperbolic}

We are now left to prove that $\pi_1(\xbarcplx)$ is hyperbolic. We appeal to Theorem \ref{thm:flatplane}. To show that there is no flat plane in the universal cover we use the fly maps from \cref{def:fly.map}. The technique is similar to the techniques used in \cite{brady_branched_1999,kropholler_finitelypresented_2018,lodha_hyperbolic_2017other}

Let $\xtilde$ be the universal cover of $\xtilde$. Begin by assuming for a contradiction that $i\colon \EE^2\to \xtilde$ is an isometric embedding of a flat plane. 

\begin{rem}\label{rem:flymapex}
	The cubical complex $\xcplx$ has 3 directions coming from the tripartite structure of $\Gamma_A$ and $\Gamma_B$. 
	There is a choice of 3 directions on $\xtilde$ given by pulling back the partition of hyperplanes from $\xcplx$. 
	Thus we obtain fly maps $f^x\colon \xtilde\to \RR^3$ for each point $x$ on $i(\EE^2)$.
\end{rem}

Recall the following definition from \cite{brady_branched_1999}.

\begin{definition}
	Given a CAT(0) cube complex $X$ and an isometric embedding $i\colon \EE^2\to X$, we say that a subset $D$ of $X$ intersects $\EE^2$ {\em transversally} at a point $p$ if there is an $\epsilon>0$ such that $N_{\epsilon}(p)\cap D\cap i(\EE^2) = \{p\}$.
\end{definition}

Let $\overline{Y}$ be the preimage of the branching locus under the branched covering map. 

\begin{theorem}\label{thm:transinter}
	There is a transverse intersection point of $i(\EE^2)$ and $\overline{Y}$. 
\end{theorem}
\begin{proof}
	We start by finding at least one point in $i(\EE^2)\cap \overline{Y}$. To find such a point, note that there is a cube $c$ in which $i(\EE^2)$ has 2-dimensional intersection. This intersection can be seen as the intersection of an affine plane in $\RR^3$ with a cube in the standard cubulation. The branching locus in this cube is depicted in Figure \ref{fig:branchcube}. If $i(\EE^2)$ does not intersect $\overline{Y}$ in $c$, then it intersects $c$ in the neighbourhood of a vertex $v$ disjoint from the branching locus. Using the fly map from \cref{rem:flymapex}, we see by \cref{cor:flylinks}, that $\lk(v, \CN(\EE^2))$ is contained in $S^0\ast S^0\ast S^0$. We can now develop into some of 8 cubes around this vertex. These 8 cubes form together as in Figure \ref{fig:8branch}, however we note that not all these cubes need exist. Within these 8 cubes there is an intersection with the branching locus. 
	
	\begin{figure}
		\centering
		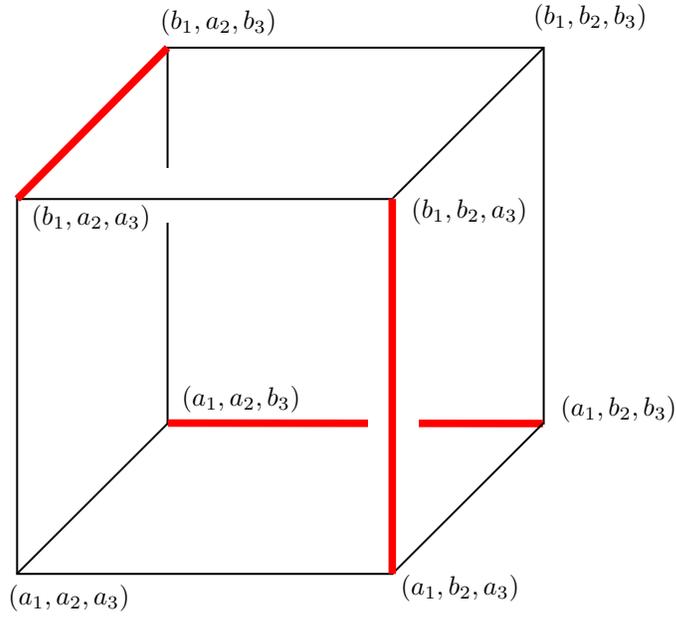
		\caption{The branching locus in one cube. The edges of the branching locus are depicted in red. }
		\label{fig:branchcube}
	\end{figure}
	
	\begin{figure}
		\centering
		\def\svgwidth{70mm}
		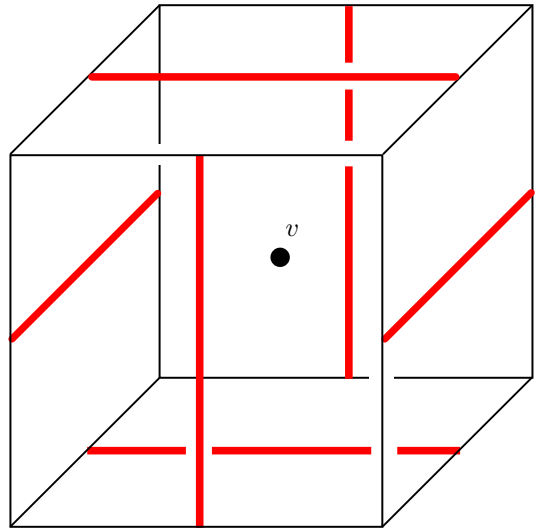
		\caption{The union of 8 cubes with the branching locus shown in bold.}
		\label{fig:8branch}
	\end{figure}
	
	In one of these 8 cubes $i(\EE^2)$ has a single point of intersection with $\overline{Y}$. We denote this cube by $c$. Note that if this point is not a transverse point of intersection, then there is an edge of $\overline{Y}$ contained in $i(\EE^2)$. 
	
	Since the intersection with $c$ is that of an affine plane. We can see that if the intersection point is in the interior of an edge, then it is a transverse intersection. 
	
	We are now left in the case that the intersection point is a vertex of the cube $c$. We begin with the case that the intersection point is a vertex which maps to $(a_1, b_2, a_3), (b_1, a_2, b_3), (a_1, b_2, b_3)$ or $(b_1, b_2, a_3)$. The link of any of these vertices is a join of a graph and a discrete set. The vertices in the discrete set correspond to $\overline{Y}$. Thus if $i(\EE^2)$ contains an edge $e$ of $\overline{Y}$ at the vertex in question, then there is a cube $c'$ sharing a face with $c$ which contains the edge $e$. This cannot happen as $i$ is an isometric embedding. 
	
	Let $O$ be the interior points of edges together with vertices mapping to $(a_1, b_2, a_3), (b_1, a_2, b_3), (a_1, b_2, b_3)$ or $(b_1, b_2, a_3)$. We will show that $i(\EE^2)$ has a point of intersection $O$. This will complete the proof by the above.  
	
	We are now reduced to the case of studying planes which contain a full edge of $\overline{Y}$ and the intersection with $c$ is not contained in $O$. The remainder of the proof is summarised in Figure \ref{fig:planeswithbranchedges}.
	
	\begin{figure}
		\centering
		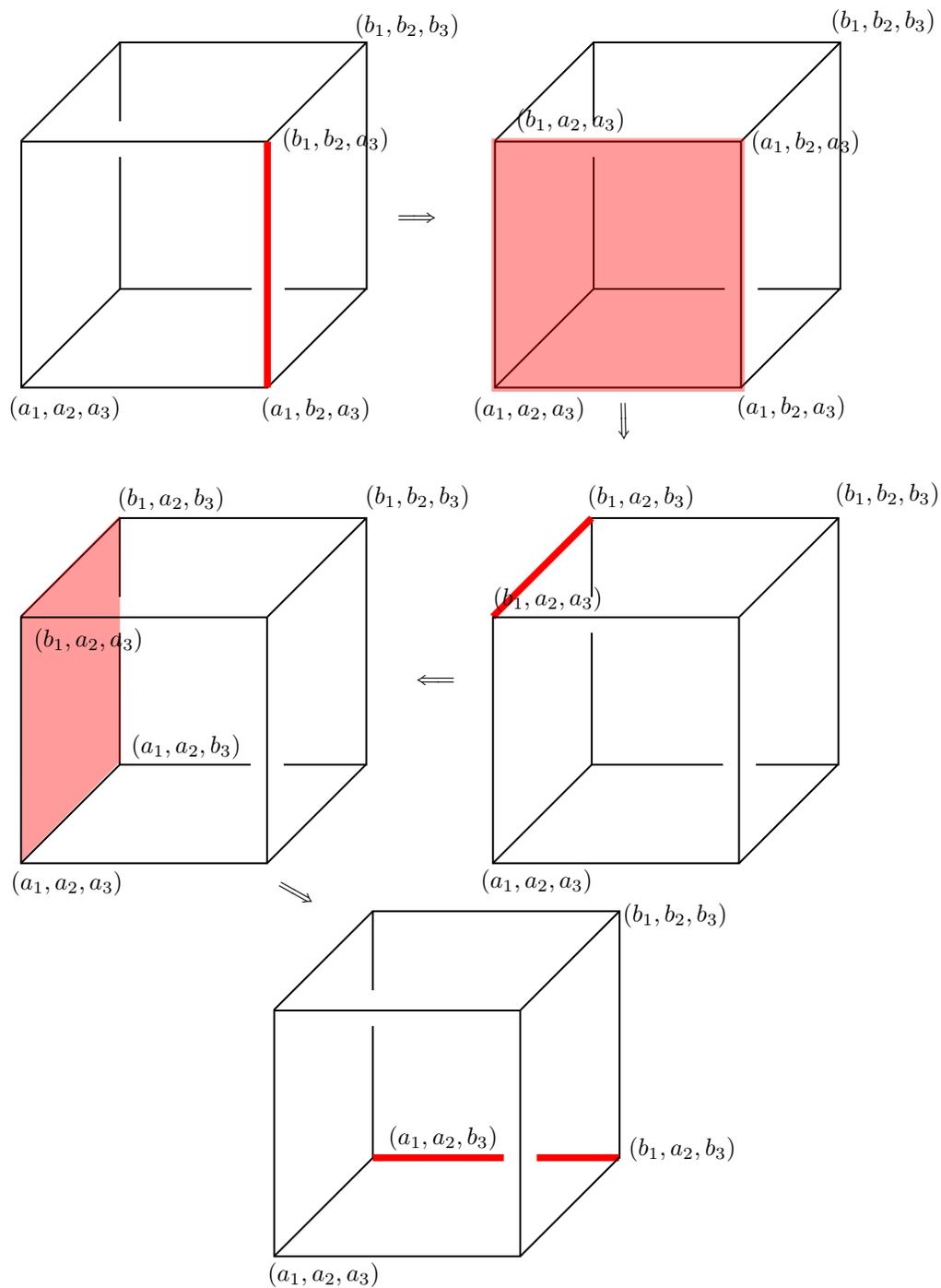
		\caption{This figure depicts the final stages of the proof of hyperbolicity. In each stage the red part of the picture is contained in the plane. In the final figure we find that containing the red line forces the plane to contain points of $O$.}
		\label{fig:planeswithbranchedges}
	\end{figure}
	
	We begin with the case of an edge between vertices that map to $(a_1, b_2, a_3)$ and $(b_1, b_2, a_3)$. There is a continuous family of embedded planes in the cube which contain this edge. However, all but one of the planes in this family intersect the cube in a point contained in $O$, this gives a transverse intersection point. 
	
	The exceptional plane is the plane containing a square with vertices mapping to $(b_1, a_2, a_3), (a_1, a_2, a_3), (a_1, b_2, a_3)$ and $(b_1, b_2, a_3)$. This intersects the edge with vertices mapping to $(b_1, a_2, a_3)$ and $(b_1, a_2, b_3)$ in the vertex mapping to $(b_1, a_2, a_3)$. 
	
	If this is not a transverse intersection point, then there is an adjacent cube in which $i(\EE^2)\cap\overline{Y}$ contains an edge of $\overline{Y}$ with vertices mapping to $(b_1, a_2, a_3)$ and $(b_1, a_2, b_3)$. 
	
	Once again there is a continuous family of flat planes containing this edge and all of them, except one, intersect in a point in $O$, giving a transverse intersection point. 
	
	The exceptional plane is the plane containing a square with vertices mapping to $(b_1, a_2, a_3), (a_1, a_2, a_3), (b_1, a_2, b_3)$ and $(a_1, a_2, b_3)$. This intersects the edge with vertices mapping to $(a_1, a_2, b_3)$ and $(a_1, b_2, b_3)$ in the vertex mapping to $(a_1, a_2, b_3)$. If this is not a transverse intersection point, then there is an adjacent cube in which $i(\EE^2)\cap\overline{Y}$ contains an edge of $\overline{Y}$ with vertices mapping to $(a_1, a_2, b_3)$ and $(a_1, b_2, b_3)$.
	
	There is a continuous family of flat planes in this cube which contain an edge with vertices mapping to $(a_1, a_2, b_3)$ and $(a_1, b_2, b_3)$. All of these flat planes will intersect in the cube in a point in $O$. This gives us a transverse point of intersection.
\end{proof}

\begin{lem}\label{lem:notransinter}
	There cannot be a transverse intersection of $i(\EE^2)$ and the preimage of the branching locus. 
\end{lem}
\begin{proof}
	Assume that there is a transverse intersection point $p$. By Corollary \ref{cor:flylinks} there is an inclusion of $\lk(p, \CN(i(\EE^2))\to S^0\ast S^0\ast S^0.$ However, this would give a loop of length 4 in the transverse direction which contradicts the choice of branched cover cf. Remark \ref{rem:4loops}. 
\end{proof}

Combining Lemma \ref{lem:notransinter} and Theorem \ref{thm:transinter} we reach our desired contradiction. Together with Theorem \ref{thm:flatplane} we arrive at the following. 

\begin{thm}\label{thm:hypbranch}
	Let $\xbarcplx$ be one of the branched covers constructed in \ref{sec:branched}. Then $\pi_1(\xbarcplx)$ is a hyperbolic group. 
\end{thm}

Putting all of this together we have the following theorem. 

\begin{thmm}
	There exists a hyperbolic group $G$ such that $G = H\rtimes \ZZ$ and $H$ is of type $FP_2$ but not finitely presented.  
\end{thmm}

\begin{proof}
	The fundamental group of $\xbarcplx$ is hyperbolic by Theorem \ref{thm:hypbranch}. Also by Corollary \ref{cor:fp2notf2}, this hyperbolic group has a subgroup which is of type $FP_2$ but not finitely presented. 
\end{proof}

\section{Uncountably many groups of type $FP_2$ which do not contain $\ZZ^2$}

We begin by defining an invariant similar to that of \cite{leary_uncountably_2015other}.

\begin{definition}
	Let $T$ be a set of words in $x_1,  \dots, x_n$. Let $G$ be a group and $S = (g_1, \dots, g_n)$ an $n$-tuple of elements in $G$.
	
	Define $$\mathcal{R}(G, S, T) = \{r\in T\mid r(S) =  1\mbox{ in }G\}.$$
\end{definition}

\begin{prop}\label{prop:countablymany}
	For a fixed set $T$ and fixed countable $G$. The invariant $\mathcal{R}(G, S, T)$ can only take countably many values. 
\end{prop}
\begin{proof}
	There are only countably many $n$-tuples of group elements. Thus once $G$ and $T$ are fixed, we only have countably many possibilities. 
\end{proof}

Now, set $L$ to be a complex with nlcp such that $\pi_1(L) = A_5$. 

We apply the proof of the preceding section to obtain a non-positively cube complex $X$ such that $\pi_1(X)$ is hyperbolic and contains a subgroup $H$ of type $FP_2$ which is not finitely presented. 

Let $\tilde{X}$ be the universal cover of $X$ and $\bar{X}$ be the cover of $X$ corresponding to the subgroup $H$. 

The Morse function on $X$ gives a real valued Morse function on $\bar{X}$. All of the ascending and descending links are connected. Hence the inclusion of $\bar{X}_{\frac{1}{2}}\to \bar{X}$ gives a surjection on the level of fundamental groups. Since $\bar{X}_{\frac{1}{2}}$ is a compact space we see that its fundamental group is finitely presented. Let $P = \langle S\mid R\rangle$ be a presentation for $\pi_1(\bar{X}_{\frac{1}{2}})$, as stated $S$ is also a generating set for $H$.

All the vertices of $\bar{X}$ map to integers thus $\bar{X}_{\frac{1}{2}}$ contains no vertices of $\bar{X}$. Let $V$ be the set of vertices in $\bar{X}$. 

\begin{lem}
	The inclusion $\bar{X}_{\frac{1}{2}}\to \bar{X}\smallsetminus V$ induces an isomorphism on fundamental groups. 
\end{lem}
\begin{proof}
	Let $U$ be the universal cover of $\bar{X}\smallsetminus V$. Let $v\in V$ be a vertex of $\bar{X}$. Since $\bar{X}$ is a locally CAT(0) cube complex we can see that $N_{\epsilon}(v)\smallsetminus\{v\}$ deformation retracts onto $\lk(v, \bar{X})$. Thus in the case that $\lk(v, \bar{X})$ is simply connected, the neighbourhood lifts to $U$. We are now concerned with the case where $\lk(v, \bar{X}) = S(S(L'))$. In this case we see that the cover of $\lk(v, \bar{X})$ is a copy of $\widetilde{S(S(L'))}$ which by \cite{leary_uncountably_2015other} is equal to $S(\widetilde{S(L')})$.
	
	We can complete $U$ to a CAT(0) cube complex which inherits a height function from $\bar{X}$. The ascending and descending links of this height function are all simply connected. Thus, the $\frac{1}{2}$-level set is connected and simply connected. So we see that the inclusion $\bar{X}_{\frac{1}{2}}\to \bar{X}\smallsetminus V$ induces an isomorphism on fundamental groups.
\end{proof}

Adding the vertices in $V$ back to $\bar{X}$ adds relations. Each time we add a vertex with link $S(S(L'))$ we are adding a new relation. Since $A_5$ is normally generated by one relation we can assume that we add one relation for each such vertex, namely, the relation obtained by coning off a normal generator of $A_5$ in $S(S(L'))$. 

We now have a presentation for $\pi_1(\bar{X})$ of the form $\langle S\mid R\cup T\rangle$, where relations in $T$ are in one to one correspondence with vertices in $\bar{X}$ with link $S(S(L'))$, Let $Y$ be the collection of these vertices. 
	
Given a subset $Z\subseteq Y$, let $T_Z$ be the subset of $T$ given by the relations corresponding to those vertices in $Z$. Let $H(Z)$ be the group given by the presentation $\langle S\mid A\cup T_Z\rangle.$

\begin{prop}\label{prop:usubz}
	If $Z\subseteq Y$, then $\mathcal{R}(H(Z), S, T) = T_Z$. 
\end{prop}
\begin{proof}
	By definition of $H(Z)$ we can see that $T_Z\subseteq \mathcal{R}(H(Z), S, T)$. 
	
	To prove the other direction we will show that for all $t\in T$, we have that $t\notin \langle\langle T\smallsetminus\{t\}\rangle\rangle.$
	
	The relation $t$ corresponds to a vertex $v$ in $\bar{X}$. 	
	Since the loop representing $t$ is trivial and normally generates $\pi_1(S(S(L')))$ we see that a small neighbourhood of $v$ lifts to the universal cover $U$. We can complete this cover by adding in the missing vertices, let $w$ be one of these vertices. Since this cover is CAT(0) we see that $U$ retracts onto $\lk(w, U) = S(S(L'))$ and thus $t$ is non-trivial. This gives us the required contradiction. 
\end{proof}

\begin{prop}\label{prop:theyarefp2}
	The groups $H(Z)$ are of type $FP_2$. 
\end{prop}
\begin{proof}
	The group $H(Z)$ is the fundamental group of $\bar{X}\smallsetminus W$ where $W$ is the set of vertices not corresponding to elements of $Z$. Taking the universal cover of this space and completing we obtain a CAT(0) cube complex $\widetilde{X(Z)}$ upon which $H(Z)$ acts. The Morse function lifts to this space. The ascending and descending links of this Morse function are simply connected, $S(L')$ or $\widetilde{S(L')}$. We can now apply Theorem\ref{bbmorse} to see that $H(Z)$ is of type $FP_2$.
\end{proof}

\begin{prop}\label{prop:theydonotcontainz2}
	The groups $H(Z)$ do not contain any copies of $\ZZ^2$. 
\end{prop}
\begin{proof}
	The action of $H(Z)$ on $\widetilde{X(Z)}$ is proper. The action is free away from the vertices with link $\widetilde{S(S(L'))}$ at these vertices the group acts like $A_5$ and thus the action is proper. We can now appeal to the Flat Torus Theorem \cite{bridson_metric_1999} to see that were there a copy of $\ZZ^2$ in $H$, then there would be an isometrically embedded flat plane in $\widetilde{X(Z)}$. Applying the proof of Theorem \ref{thm:hypbranch} we can see that no flat plane exists. 
\end{proof}

\begin{thmm}
	There are uncountably many groups of type $FP_2$ none of which contains a Baumslag-Solitar group or infinite torsion subgroups. 
\end{thmm}
\begin{proof}
	Since $T$ is infinite there are uncountably many subsets $T_Z$. By Propositions \ref{prop:countablymany} and \ref{prop:usubz} we see that there are uncountably many groups in this family. 
	
	All of these groups are of type $FP_2$ by Proposition \ref{prop:theyarefp2}. They also do not contain a copy of $\ZZ^2$ by Proposition \ref{prop:theydonotcontainz2}. 
	
	We are now left in the case that these groups contain $BS(1, n)$ for $n\neq \pm 1$. By \cite[pg. 439]{bridson_metric_1999} groups acting properly and semi-simply on CAT(0) spaces do not contain $BS(m, n)$ for $|n|\neq |m|$. 
	
	We now prove the statement about infinite torsion subgroups. 
	Each of the groups constructed acts on a locally finite CAT(0) cube complex. 
	Thus any torsion group fixes a point by \cite[Theorem 5.1]{sageev_ends_1995}. 
	However, the group action is proper so point stabilisers are finite. 
\end{proof}

\appendix

\section{Embedding cube complexes into products of trees and fly maps \\ {\small (by Robert Kropholler and  Federico Vigolo)}}
\label{sec:embedding_in_trees}

This appendix is devoted to justify the construction of fly maps used in Section~\ref{sec:flymaps}.
Many parts of this appendix are well-known, we include them to give the reader a better understanding of how fly maps are defined.  
The definition of a cube complex with $n$ directions is given in \cref{def:directions}.

\begin{definition}\label{def:fly.map} 
	Let $X$ be a cube complex with $n$ directions. A {\em fly map} $f\colon X\to \RR^n$ for an embedded flat $i\colon \EE^k\to X$ is a map $f\colon X\to \RR^n$ obtained as a composition
	\[
	X\xrightarrow{\ \iota\ } T_1\times\cdots \times T_n \xrightarrow{f_1\times\cdots\times f_n} \RR^n
	\]
	where the maps $f_j\colon T_j\to \RR$ are cubical maps that restrict to isometries on $L_j = p_j(\iota(i(\EE^k)))$. 
\end{definition}

In this appendix we will define all the relevant terms and show that these maps always exist.

Let $X$ be a CAT(0) cube complex. A \emph{hyperplane} $\hpl$ of $X$ is an equivalence class of parallel edges of $X$. We can identify a hyperplane $\hpl$ with the CAT(0) cube complex spanned by the midpoints of the edges in $\hpl$ (\emph{i.e.} the smallest convex set containing them). Every hyperplane $\hpl$ disconnects $X$ into two \emph{half spaces} $\hsp$ and $\hsp^*$. 
We denote the set of hyperplanes in $X$ by $\Hpl(X)$ and the set of half spaces by $\Hsp(X)$.

A {\em pocset} is a partially ordered set $(S,<)$ with an involution $*\colon S\to S$ such that 
\begin{enumerate}
	\item $\fks \neq \fks^*$ for all $\fks\in S$ and $\fks$ and $\fks^*$ are incomparable, 
	\item if $\fks<\fkt$, then $\fkt^*<\fks^*$. 
\end{enumerate}
We will denote the pocset simply by $S$ if the ordering $<$ is clear by the context.

The set of half spaces $\Hsp(X)$ of a CAT(0) cube complex comes naturally equipped with a pocset structure, where the involution sends a half space $\hsp$ to its complement $\hsp^*$ and the partial ordering is given by the inclusion.

Generalising the idea of ultrafilters on the subsets of an index set, an {\em ultrafilter} on a pocset $S$ is a subset $\CU\subseteq S$ such that 
\begin{enumerate}
	\item for each $\fks\in S$ exactly one of $\fks$ or $\fks^*$ is in $\CU$ (completeness), 
	\item if $\fks\in \CU$ and $\fks<\fkt$, then $\fkt\in \CU$ (consistency). 
\end{enumerate}

An ultrafilter satisfies the {\em descending chain condition} (DCC) if every descending chain $\fks_1>\fks_2>\cdots$ must terminate in finitely many steps. 

Given a pocset $S$ we can construct a cube complex $X(S)$ as follows:
the vertices of $X(S)$ are the ultrafilters satisfying the DCC; two ultrafilters $\CU_1, \CU_2$ are joined by an edge if and only if they differ by precisely two elements: $\CU_1\triangle \CU_2 = \{\fks, \fks^*\}$. The graph thus obtained is the $1$\=/skeleton of $X(S)$, and we then add higher dimensional cubes to this graph whenever their $1$\=/skeleton is present. The following is well\=/known:

\begin{thm}[Sageev \cite{sageev_ends_1995}]
	The cube complex $X(S)$ is CAT(0). 
\end{thm}

Hyperplanes of the cube complex $X(S)$ are in one\=/to\=/one correspondence with pairs of conjugated elements of $S$ i.e. fixed sets $\{\fks,\fks^*\}\subseteq S$. All the edges determined by $\CU_1\triangle \CU_2 = \{\fks, \fks^*\}$ are parallel and determine a hyperplane $\hat\fks$ of $X(S)$. The (vertices of the) half spaces determined by $\hat\fks$ can be described as
\begin{align*}
\fks &\coloneqq\bigbrace{\CU\bigmid\CU\subseteq S\text{ ultrafilter with DCC},\ \fks\in\CU} \\
\fks^* &\coloneqq\bigbrace{\CU\bigmid\CU\subseteq S\text{ ultrafilter with DCC},\ \fks^*\in\CU}.
\end{align*}
In particular, the pocset of hyperspaces $\Hsp(X(S))$ is equal to $S$ itself.

Vice versa, since $\Hsp(X)$ is a pocset one can apply Sageev's construction to it and one can prove the following theorem.

\begin{theorem}[Roller duality \cite{sageev_cat0_2014}]
	Let $X$ be a CAT(0) cube complex. Then 
	\[
	X(\Hsp(X)) = X.
	\]
\end{theorem}

A \emph{subpocset} of $S$ is a subset $A\subseteq S$ with $A=A^*$ equipped with the partial ordering induced by $S$. We say that a pocset $S$ {\em splits}, denoted $S=A\dotsqcup B$, if there exist two non\=/empty subpocsets $A,B\in S$ such that
\begin{enumerate}
	\item the set $S$ is the disjoint union $A\sqcup B$,
	\item every $\fka\in A$ is incomparable with every $\fkb\in B$.
\end{enumerate}
It is proved in \cite{caprace_rank_2011} that $X(S)=X(A)\times X(B)$ if and only if $S=A\dotsqcup B$.

\

We need to study $\ell^1$\=/isometric embeddings of cubes complexes into finite products of trees (see also \cite{bandelt_embedding_1989} and \cite{ChHa13}). It is shown in \cite{haglund_isometries_2007} that the $\ell^1$\=/distance between two vertices $v,w\in X$ is equal to the number of hyperplanes separating them. That is, 
\[
d_{\ell^1}(v,w)=\abs{\bigbrace{\fkh\in \Hsp(X)\mid v\in\fkh,\ w\notin\fkh}}. 
\]

\begin{definition}\label{def:directions}
	A choice of {\em $n$ directions} on a CAT(0) cube complex $X$ is a decomposition of the set of hyperplanes as a disjoint union $\Hpl= \Hpl_1\sqcup\dots\sqcup\Hpl_n$ such that for every $i=1,\ldots,n$ any two hyperplanes in $\Hpl_i$ are disjoint. Equivalently, a choice of $n$ directions is a decomposition of the pocset of half\=/spaces $\Hsp$ as a disjoint union of $n$ subpocsets $\Hsp=\Hsp_1\sqcup\cdots\sqcup\Hsp_n$ such that for all $\hsp_1, \hsp_2\in \Hsp_i$ either $\hsp_1 \leq \hsp_2$ or $\hsp_1^*\leq \hsp_2^*$ or $\hsp_1^*\leq \hsp_2$ or $\hsp_1\leq \hsp_2^*$.
\end{definition}

\begin{rem}
	If $\Gamma_A$ and $\Gamma_B$ are $n$\=/coloured flag complexes, then the universal cover of the associated cube complex with coupled links $\xcplx$ (as defined in \cite{kropholler_new_2018}) is CAT(0) and has a natural choice of $n$ directions. Namely, $\widehat \CH_i$ is the family of hyperplanes that are perpendicular to edges in the $i$\=/th coordinate. 
\end{rem}

\begin{rem}
	The decomposition $\Hsp=\Hsp_1\sqcup\cdots\sqcup\Hsp_n$ does not need to be a splitting of the pocset $\Hsp$. Indeed, in general there will be comparable half spaces $\hsp\in\Hsp_i$ and $\hsp'\in\Hsp_j$ for some $i\neq j$.
\end{rem}

\begin{rem}
	Recall that the \emph{crossing graph} of a CAT(0) cube complex is the graph whose vertices are the hyperplanes and where two vertices are joined by an edge if the corresponding hyperplanes intersect.
	Choosing $n$ directions on $X$ is equivalent to choosing an $n$\=/colouring for the crossing graph. 
\end{rem}

We define a map to a product of trees as follows. Let $X$ be a cube complex with $n$ directions. Let $T_i\coloneqq X(\mathcal{H}_i)$ be the cube complex obtained applying Sageev's construction to the pocset $\CH_i$. It is easy to show that $T_i$ is a tree for each $i$.

Let $Y = \prod_{i=1}^nT_i$. Then $\mathcal{H}(Y) = \mathcal{H}_1\dotsqcup\dots\dotsqcup \mathcal{H}_n$. In particular, $\Hsp(X)$ and $\Hsp(Y)$ coincide as sets. We will show that the map defined by the identity from $\mathcal{H}(X)\to \mathcal{H}(Y)$ induces a cubical map $\iota\colon X\to Y$.

\begin{lem}
	An ultrafilter $\mathcal{U}$ on $\mathcal{H}(X)$ satisfying the descending chain conditions defines an ultrafilter on $\mathcal{H}(Y)$ satisfying the descending chain condition. 
\end{lem}
\begin{proof}
	When seen as a subset of $\Hsp(Y)$, the set $\mathcal{U}$ will still satisfy the completeness axiom as the complementation structure on $\mathcal{H}(X)$ is the same as that on $\mathcal{H}(Y)$. We must check the consistency axiom, but it is clear that being comparable is weaker in the set $\mathcal{H}(Y)$ than in $\mathcal{H}(X)$, so the consistency of $\CU$ in $\CH(Y)$ is trivially implied by the consistency of $\CU$ in $\CH(X)$. 
	
	The above argument also implies that any descending chain in $\CU$ terminates, as the descending chains in $\mathcal{H}(Y)$ are descending chains in $\mathcal{H}(X)$ as well.  
\end{proof}

It follows from the above that the identity map $\CH(X)\to\CH(Y)$ induces an injection between the vertex sets of the associated cube complexes $\iota\colon V(X)\hookrightarrow V(Y)$\textemdash here we are implicitly using Roller's duality on Sageev's construction $X=X\bigparen{\CH(X)}$. The injection $\iota$ extends to a cubical injection $\iota\colon X\to Y$. In fact, something stronger is true:

\begin{prop}\label{prop:iota.ell.1.embedding}
	The map $\iota\colon V(X)\hookrightarrow V(Y)$ extends to an $\ell^1$\=/embedding of cube complexes $\iota\colon X\hookrightarrow Y$. Moreover, taking the preimage $\iota^{-1}$ induces a bijective correspondence between hyperplanes of $Y$ and $X$.
\end{prop}
\begin{proof}
	The key fact is that in Sageev's construction hyperplanes of $X(S)$ correspond to pairs of conjugate elements $\braces{\fks,\fks^*}$. In our setting, the pocsets $\Hsp(X)$ and $\Hsp(Y)$ do not just coincide as sets: they also have the same involution operation. It follows that there is a natural correspondence between the hyperplanes of $X$ and $Y$ as they are given in both cases by couples $\braces{\hsp,\hsp^*}$.
	
	Recall that, according to Sageev's construction, two ultrafilters $\CU$ and $\CU'$ with the DCC are linked by an edge if and only if they differ by two elements $\CU\triangle\CU=\braces{\fkh,\fkh^*}$ (\emph{i.e.} two points are linked by an edge if and only if they are separated by a unique hyperplane). From the discussion above, this condition is independent on whether we are looking at the pocset $\Hsp(X)$ or $\Hsp(Y)$, therefore the map $\iota$ extends to a cubical map.
	
	The map $\iota$ is an $\ell^1$\=/embedding because the $\ell^1$\=/metric is equal to the number of hyperplanes separating them (a hyperplane $\braces{\hsp,\hsp^*}$ separates two points $\CU$ and $\CU'$ in $X=X(\Hsp(X))$ if and only if $\braces{\hsp,\hsp^*}\subseteq\CU\triangle\CU'$). This condition is again preserved when passing from $\Hsp(X)$ to $\Hsp(Y)$. 
	
	Let $\hpl_X$ be the hyperplane of $X$ corresponding to the pair $\braces{\hsp,\hsp^*}$ (seen as the span of middle points of parallel edges) and let $\hpl_Y$ be the hyperplane of $Y$ determined by the same couple $\braces{\hsp,\hsp^*}$. To prove the `Moreover' part of the statement we need to show that the preimage $\iota^{-1}(\hpl_Y)$ coincides with $\hpl_X$. This is another easy consequence of the above discussion. Indeed, an edge $e$ between two vertices in $X$ uniquely determines a hyperplane $\hpl_X=\braces{\hsp,\hsp^*}$ and its image $\iota(e)$ will uniquely determine the corresponding hyperplane $\hpl_Y=\braces{\hsp,\hsp^*}$. As $\hpl_Y$ is the hyperplane spanned be the midpoints of all (and only) the edges crossing it, it follows that the pre\=/image $\iota^{-1}\paren{\hpl_Y}$ intersects all (and only) the edges crossing $\hpl_X$ and it hence coincides with the hyperplane $\hpl_X$.
\end{proof}

\begin{rem}
	More concretely, the map $\iota$ is obtained as follows: the choice of $n$ directions $\Hpl_1\sqcup,\cdots\sqcup\Hpl_n$ identifies $n$ trees $T_1,\ldots, T_n$, and $\iota$ is then defined by sending a point $x\in X$ to the point in $Y=T_1\times\cdots\times T_n$ whose $i$\=/th coordinate is determined by the relative position of $x$ with respect to the hyperplanes in $\Hpl_i$. This can be visualised by seeing that crossing a hyperplane in $\Hpl_i$ corresponds to crossing the corresponding edge of $T_i$. The fact that this procedure is well defined depends on the fact that the cube complex $X$ is CAT(0). Sageev's construction is a useful tool formalise this argument.
\end{rem}

Recall that the cubical neighbourhood $N(A)$ of a set $A$ in a cube complex $X$ is the smallest subcomplex of $X$ containing $A$. Given a hyperplane $\hpl\subseteq X$, its cubical neighbourhood $N(\hpl)$ is isometric to $\hpl\times[0,1]$. We will denote by $\interior{N}(\hpl)$ the interior of the cubical neighbourhood $N(\hpl)$.

Let $X$ be a cube complex with $n$ directions. Let $Y=T_1\times\cdots\times T_n$ be the associated product of trees and let $p_i\colon Y\to T_i$ be the projection to the $i$-th tree. Every edge in $T_i$ identifies a hyperplane $\hpl_{T_i}$ in $\Hpl(T_i)=\Hpl_i$ and hence a hyperplane $\hpl_X$ in $X$. We have the following:

\begin{lem}\label{lem:preimage.of.open.edges}
	Let $\interior{e}$ be the interior of an edge of $T_i$ and let $\hpl_i\in \Hpl_i$ be the corresponding hyperplane in $X$. Then $(p_i\circ \iota)^{-1}(\interior{e}) = \interior{\mathcal{N}}(\hpl)$.
	
	In particular, for every $x\in \interior e$ the pre\=/image $(p_i\circ \iota)^{-1}(x)$ is a convex subset of $X$ isometric to $\hpl$.
\end{lem}

\begin{proof}
	Since $Y$ is a direct product, it is clear that $p_i^{-1}(\interior e)$ coincides with $\interior N(\hpl_Y)\subseteq Y$. Now the lemma follows easily from Proposition~\ref{prop:iota.ell.1.embedding} because we proved that the pre\=/image $\iota^{-1}(\hpl_Y)$ coincides with $\hpl_X$.
\end{proof}

\begin{rem}
	In Lemma~\ref{lem:preimage.of.open.edges} it is important to restrict to the interior of the edge $e$. Indeed, the pre\=/image of the extremal points will not be contained in $N(\hpl_X)$ in general. 
\end{rem}

\begin{lem}\label{linelem}
	Let $i\colon \EE^k\hookrightarrow X$ be an isometric embedding of an Euclidean flat. Then $p_j\circ\iota\circ i(\EE^k)\to T_j$ has image contained in a geodesic. 
\end{lem}
\begin{proof}
	If the image is not contained in a geodesic then it contains a branching point. That is, $p_j\circ\iota\circ i(\EE^k)\to T_j$ must intersect the interior three edges $e_1,e_2$ and $e_3$ of $T_i$ sharing a vertex $v\in T_i$. Let $x_r$ be a point in $\interior e_r\cap \bigparen{p_j\circ\iota\circ i(\EE^k)}$ for $r=1,2,3$. Then, by Lemma~\ref{lem:preimage.of.open.edges}, the pre\=/image $(p_j\circ\iota)^{-1}(x_r)$ will be a convex subset of $X$ that separates $X$ in two half spaces. It follows that $(p_j\circ\iota\circ i)^{-1}(x_r)$ is a convex subset of $\EE^k$ that separates $\EE^k$ and it must hence contain a hyperplane $E_r\cong\EE^{k-1}$\textemdash here $E_r$ is a hyperplane in the usual Euclidean sense: not as a cube complex. (Since $(p_j\circ\iota)^{-1}(x_r)$ is a parallel copy of a hyperplane $\hpl_r$ in $\interior N(\hpl_r)$, it is actually easy to show that $(p_j\circ\iota\circ i)^{-1}(x_r)$ is itself a hyperplane in $\EE^k$.)
	
	For every $r=1,2,3$ we can consider the component $H_r$ of $\EE^k\smallsetminus E_r$ \emph{not} containing the pre\=/image of the vertex $v$, and this yields to a contradiction because the spaces $H_r$ would form a facing triple. That is, they would be three disjoint halfspaces in $\EE^k$, and it is easy to see that there is no such triple in the Euclidean space.
\end{proof}

The existence of fly maps as by Definition~\ref{def:fly.map} follows easily from Lemma~\ref{linelem}. The last claim we have to prove is the following:

\begin{prop}\label{prop:fly.maps.l1.isometry.appendix}
	Let $\CN\bigparen{i(\EE^k)}\subseteq X$ be the cubical neighbourhood of an embedded flat, then any fly map $f$ restricts to a cubical $\ell^1$\=/embedding $\CN\bigparen{i(\EE^k)}\hookrightarrow \RR^n$. 
\end{prop}

\begin{proof}
    Recall that $L_j = p_j(\iota(i(\EE^k)))$. Enlarging it if necessary, we can assume that $L_j$ is a subcomplex of $T_j$ for every $j=1,\ldots,n$.
	It is clear from the definition that any fly map $f$ induces a cubical isometric embedding $\prod_{j=1}^n L_j\to\RR^n$ with respect to both the Euclidean and $\ell^1$\=/metrics induced by their cubical structures.
	
	Note that $\prod_{j=1}^n L_j\subseteq\prod_{j=1}^n T_j$ is a convex subset with respect to both the Euclidean and the $\ell^1$\=/metric, therefore the metric induced on $\prod_{j=1}^n L_j$ from the Euclidean (resp. $\ell^1$) metric of $\prod_{j=1}^n T_j$ coincides with the Euclidean (resp. $\ell^1$) metric induced from its cube complex structure. 
	
	It follows that the map induced from $f$ on $\prod_{j=1}^n L_j$ is an $\ell^1$\=/embedding also with respect the subspace metric. The statement now follows trivially from Proposition~\ref{prop:iota.ell.1.embedding} by noting that the image of $\CN\bigparen{i(\EE^k)}$ under $\iota$ is contained in $\prod_{j=1}^n L_j$.
\end{proof}

\begin{cor}\label{cor:flylinks}
	A fly map $f^x$ induces for every cube $c\in \CN\bigparen{i(\EE^k)}$ an embedding $\lk\bigparen{c, \CN\bigparen{i(\EE^k)}}\to \SS^0\ast\dots\ast \SS^0$.  
\end{cor}

\bibliographystyle{plain}
\bibliography{bibshort,otherref}

\end{document}